\documentclass[a4paper,10pt]{amsart}
\usepackage{txfonts,amsfonts}
\usepackage[arrow,matrix]{xy}
\usepackage{amsmath,amssymb,amsthm,amscd,bbm,mathrsfs,dsfont,bm}
\usepackage{extarrows}
\usepackage{latexsym}
\usepackage{graphicx}
\usepackage{color}
\usepackage{xy}
\usepackage{leftidx}
\input xy
\input xypic
\xyoption{all}

\numberwithin{equation}{section}

\theoremstyle{plain}
\newtheorem{theorem}{Theorem}[section]
\newtheorem{corollary}[theorem]{Corollary}
\newtheorem{lemma}[theorem]{Lemma}

\newtheorem{proposition}[theorem]{Proposition}

\theoremstyle{definition}
\newtheorem{definition}[theorem]{Definition}

\newtheorem{example}[theorem]{Example}

\theoremstyle{remark}
\newtheorem{remark}[theorem]{Remark}

\DeclareMathOperator{\Hom}{Hom}

\DeclareMathOperator{\End}{End}

\DeclareMathOperator{\id}{\rm id }
\DeclareMathOperator{\sh}{\rm Sh }

\DeclareMathOperator{\lex}{\rm lex }

\DeclareMathOperator{\gr}{\rm gr }

\newcommand{\hh}{\textbf{\rm h}}

\newcommand{\N}{\mathcal{N}}

\renewcommand{\L}{\mathbb{L}}

\begin{document}
	
	\title{The Adams operators on connected graded Hopf algebras}
	
     \author{Y.-Y. Li, \; G.-S. Zhou }

  \address{\rm Li \newline \indent
  	School of Mathematics and Statistics, Ningbo University, Ningbo 315211, China
  	\newline \indent E-mail: 1951897284@qq.com
  	\newline\newline
  	\indent Zhou \newline \indent
  	School of Mathematics and Statistics, Ningbo University, Ningbo 315211, China
  	\newline \indent E-mail: zhouguisong@nbu.edu.cn}
	
\begin{abstract}
The Adams operators on a Hopf algebra $H$ are the convolution powers of the identity map of $H$. They are also called Hopf powers or Sweedler powers. It is a natural family of operators on $H$ that contains the antipode. We study the linear properties of the Adams operators when $H=\bigoplus_{m\in \mathbb{N}} H_m$ is connected graded. The main result is that for any of such $H$,  there exist a PBW type homogeneous basis and a natural total order on it such that the restrictions $\Psi_n|_{H_m}$ of the Adams operators are simultaneously upper triangularizable with respect to this ordered basis. Moreover, the diagonal coefficients are determined in terms of $n$ and a combinatorial number assigned to the basis elements. As an immediate consequence, we obtain a complete description of the characteristic polynomial of  $\Psi_n|_{H_m}$, both on eigenvalues and their multiplicities, when $H$ is locally finite and the base field is of characteristic zero. It recovers the  main result of the paper \cite{AgLa} by Aguiar and Lauve, where the approach is different from ours.
\end{abstract}
	


	
	\maketitle

	
	
	
	
\section*{Introduction}

\newtheorem{maintheorem}{\bf{Theorem}}
\renewcommand{\themaintheorem}{\Alph{maintheorem}}
	
A \emph{connected ($\mathbb{N}$-)graded Hopf algebra} is a Hopf algebra $H$ equipped with a grading $H=\bigoplus_{m\in \mathbb{N}} H_m$ such that all of its structure maps preserve the grading and the zeroth component $H_0$ is one-dimensional. The classical examples of such Hopf algebras are the universal enveloping algebras of positively graded Lie algebras. In recent years, many nontrivial new examples appear in the classification results on connected Hopf algebras of finite Gelfand-Kirillov dimension \cite{BGZ,BZ2, NWW, WZZ, Wang, Zh}. The most fruitful source of interesting examples is of combinatoric nature. A fundamental principle observed by  Joni and Rota \cite{JR} is that discrete structures may give rise to natural Hopf algebras whose (co)multiplications encode the (dis)assemble of those structures. Hopf algebras arised in this perspective always turn out to be connected graded (with each homogeneous component being finite dimensional). Celebrated examples include the Hopf algebra of symmetric functions \cite{Mac}, of quasi-symmetric functions \cite{Ges}, of noncommutative symmetric functions \cite{GKLL}, of (un)labeled graphs \cite{Schm}, of permutations of Malvenuto and Reutenauer \cite{MaRe}, of rooted trees of Connes and Kreimer \cite{CoKr}, of planar binary trees of Loday and Ronco \cite{LoRo}, and so on. We refer to  \cite{GrRe} for more examples and references. 	
	
The convolution powers $\Psi_n$ of the identity map of a Hopf algebra $H$ is a natural family of operators on $H$ that contains the antipode. In the case that $H=k[G]$ is a group algebra, $\Psi_n$ is the linear extension of the $n$-th power map on $G$. These maps are called \emph{Adams operators} in this paper, follow the  terminology used in \cite{AgLa,Cart,Lod}, which is reminiscent of the one in topological K-theory introduced  by Adams \cite{Adam}. Other common terminology for them are Hopf powers, Sweedler powers and characteristic endomorphisms. Explicitly Adams operators  seem to appear first in \cite{TaOo}. They have been actively used in the study of Hochschild and cyclic (co)homology of commutative algebras \cite{GeSc,Lod}, and in the program to carry over to Hopf algebras some of rich theory of groups \cite{KSZ,LMS}. 
	
In this paper, we study the behavior of the Adams operators when $H$ is connected graded. General structure results may be available because in some sense being connected graded is a quite restrictive condition on Hopf algebras. Aguiar and Lauve have done an excellent work in this line of research \cite{AgLa}. They give a complete description of the characteristic polynomial -- both eigenvalues and their multiplicities -- for the restrictions $\Psi_n|_{H_m}$ of Adams operators on homogeneous components of $H$, provided that $H$ is locally finite and the base field is of characteristic zero. As consequences, they obtain many combinatorial descriptions of the trace of $\Psi_n|_{H_m}$, particularly for those $H$ of combinatoric nature. Another notable work in this line of research is a paper by  Diaconis, Pang and  Ram \cite{DPRa}, where the Adams operators on connected graded Hopf algebras are employed to induce finite Markov chains (e.g. card shuffling and rock-breaking) and their stationary distributions.

It is well-known to experts that for any connected graded Hopf algebra $H$ that is commutative or cocommutative, the family  $(\Psi_{n}|_{H_m})_{n\in \mathbb{Z}}$ of restrictions of the Adams operators 
are simultaneously diagonalizable when the base field is of characteristic zero. However, even $\Psi_2|_{H_3}$ itself may not be diagonalizable in general, see Example \ref{counterexample} for a counterexample. So it is natural to ask that what kind of weaker linear properties can these maps share.  This paper is devoted to present an answer to this question. The main result of this paper may be phrased as follows, which is considered as a plain version of the combination of Theorem \ref{Adams operators} and Theorem \ref{PBW-basis}.

\begin{maintheorem}
For any connected graded Hopf algebra $H$, there exist a PBW type homogeneous basis and a natural total order on it such that the restrictions $\Psi_n|_{H_m}$ of the Adams operators are simultaneously upper triangularizable with respect to this ordered basis. Moreover, the diagonal coefficients are determined in terms of $n$ and a combinatorial number assigned to the basis elements.
\end{maintheorem}

Roughly speaking, by a PBW type basis of an algebra  we mean a basis consists of specific sorted products of some ordered family of elements of the algebra. In this paper, we assume for each element of the given family an upper bound on the number of the occurrences of itself in sorted products. As an immediate consequence of the above theorem (more precisely, Theorem \ref{Adams operators} and Theorem \ref{PBW-basis}), we obtain  in Corollary \ref{characteristic polynomial} a complete description of the characteristic polynomial of $\Psi_n|_{H_m}$, which recovers the main result of the paper \cite{AgLa}. The PBW type bases that are workable for our purpose shall satisfy several desirable conditions as listed in Theorem \ref{Adams operators}. The existence of such a basis for every connected graded Hopf algebra is claimed in Theorem \ref{PBW-basis}. Indeed, it has been established essentially in the paper \cite{ZSL3} by Lu, Shen and the second-named author by means of a combinatorial method based on Lyndon words, which is originally developed by Kharchenko \cite{Kh}. 

The paper is organized as follows. In Section 1, we recall the basic notions and facts related to the Adams operators and connected graded Hopf algebras. In Section 2, we study the behavior of the Adams operators on connected graded Hopf algebras. It is the core part of this paper, and there the main results are presented (without proof) as Theorem \ref{Adams operators} and Theorem \ref{PBW-basis}. Sections 3 and 4 are devoted to prove Theorem \ref{Adams operators}. Some technical formulas, which have their own right of interests, are established in Section 3. Section $5$ is devoted to prove Theorem \ref{PBW-basis}.
In the final section, we discuss the Hopf algebra of permutations by examining the basic ideas of this paper.

Throughout this paper, we work over a fixed field (of arbitrary characteristic) denoted by $k$. All vector spaces, algebras, coalgebras and unadorned tensor products are taken over $k$. The set of nonzero elements of $k$ is denoted by $k^\times$. We write  $\mathbb{N}$ and $\mathbb{Z}$ for the set of natural numbers and integers respectively, and let $\mathbb{N}_\infty:=\mathbb{N}\cup \{\infty\}$. By convention, $\mathbb{N}$ contains the number $0$.

\section{Preliminaries}
\label{Basic-definitions}

In this section, we recall the basic notions as well as some of their elementary properties that will be used in the sequel. Most of them are standard, for which we refer to \cite{HS,Mont}.  For the material on the Adams operators, we refer to \cite{AgLa,Pat} and \cite[Section 4.5]{Lod}.

\subsection{Definition of the Adams operators}

Let $H$ be a  bialgebra. We denote by $\mu_H$, $\Delta_H$, $\eta_H$ and $\varepsilon_H$ the  multiplication map,  comultiplication map,  unit map and counit map of $H$ respectively.

Let $\End(H)$ be the space of all linear endomorphisms  of  $H$. The \emph{convolution product} of $f,~g \in \End(H)$ is defined to be the linear endomorphism
\begin{eqnarray*}\label{convolution-prouct}
f*g := \mu_H\circ (f\otimes g) \circ \Delta_H.
\end{eqnarray*}
This turns $\End(H)$ into an associative algebra with unit $\eta_H\circ \varepsilon_H.$ The convolution  power $$\Psi_n:=(\id_H)^{*n}$$ of the identity map $\id_H\in \End(H)$ is  defined for any integer $n\geq 0$.  They are known as the \emph{Adams operators} of $H$. Other common terminology for these operators are Hopf powers, Sweedler powers and characteristic endomorphisms. 
By convention, $\Psi_0=\eta_H\circ \varepsilon_H.$ 

The bialgebra $H$ is by definition a \emph{Hopf algebra}  if  $\id_H$ is convolution-invertible. In this case, 
the \emph{antipode}  of $H$, denoted by $S_H$, is defined to be the convolution inverse of $\id_H$:
\[
S_H *\id_H=\id_H*S_H= \eta_H\circ \varepsilon_H.
\]
Moreover, $(\id_H)^{*n}:=(S_H)^{*-n}$ are defined for any negative integer $n<0$. They are also called the Adams operators of $H$, and are denoted by $\Psi_n$. Note that $\Psi_{-1}=S_H$.

\subsection{Connected graded Hopf algebras}
\label{connected graded Hopf algebras}
Let $\Gamma$ be a nontrivial free abelian monoid with neutral element denoted by $0$. So for each $\gamma\in \Gamma$, there are only finitely many pairs $(\alpha,\beta) \in \Gamma \times \Gamma$ such that $\gamma= \alpha+\beta$. Let $N(\gamma)$ be the largest number  such that $\gamma$ can be expressed as $\gamma= \alpha_1+\cdots +\alpha_{N(\gamma)}$ with $\alpha_i\in \Gamma \backslash\{0\}$. It is easy to see that this number is well-defined. By convention, we write $N(0):=0$.  For a $\Gamma$-graded vector space $A=\bigoplus_{\gamma\in \Gamma}A_\gamma$, we write $A_+:= \bigoplus_{\gamma\neq 0}A_\gamma$.

A \emph{$\Gamma$-graded bialgebra} is a bialgebra $H$ together with a $\Gamma$-grading $H=\bigoplus_{\gamma\in \Gamma}H_\gamma$  such that 
\begin{eqnarray*}
	& H_\alpha H_\beta \subseteq H_{\alpha+ \beta} \quad \text{for all } \alpha, \beta\in \Gamma, & \\
	&\Delta_H(H_\gamma) \subseteq \sum_{\alpha+\beta=\gamma} H_\alpha\otimes H_\beta \quad \text{for all } \gamma\in \Gamma,& \\
	& 1_H\in H_0 \quad \text{and} \quad \sum_{\gamma\neq 0} H_\gamma \subseteq \ker(\varepsilon_H).&
\end{eqnarray*}
If $H$ is in addition a Hopf algebra, then the antipode $S_H:H\to H$ is necessarily a graded linear map (\cite[Corollary 5.1.4]{HS}), and we refer to such $H$ as a \emph{$\Gamma$-graded Hopf algebra}. 

A $\Gamma$-graded bialgebra  $H=\bigoplus_{\gamma\in \Gamma}H_\gamma$ is called \emph{connected} if $H_0$ is one-dimensional, namely it is connected as a $\Gamma$-graded algebra.  Such $H$ is necessarily a Hopf algebra with antipode 
\begin{eqnarray} \label{antipode}
	S_H=\sum_{n\geq 0} (\eta_H\circ \varepsilon_H -\id_H)^{* n}.
\end{eqnarray}
One may refer to \cite[Corollary 5.2.12]{HS} for this basic result. The sum in (\ref{antipode}) is actually finite when evaluated on any $x\in H$. More precisely,  $(\eta_H\circ \varepsilon_H -\id_H)^{* n}(x)=0$ for any $x\in H_\gamma$ and $n>N(\gamma)$. Since the convolution product of any two of graded linear endomorphisms of $H$ is again graded linear, the expression (\ref{antipode}) tells us explicitly that $S_H$ is also a graded linear map.

Let $H$ be a connected $\Gamma$-graded Hopf algebra. Then the coradical of $H$ is $H_0$. Therefore, the coradical filtration $\{H(n)\}_{n\in \mathbb{N}}$ of $H$  is a Hopf algebra filtration. It is easy to see that each layer $H(n)$ is a homogeneous subspace of $H$, so the associated graded Hopf algebra
\[
\gr(H) := \bigoplus_{n\in \mathbb{N}} H(n)/H(n-1)
\]is naturally  a connected $\Gamma$-graded Hopf algebra with 
\[
\gr(H)_\gamma = \bigoplus_{n\in \mathbb{N}} \Big(H(n)/H(n-1)\Big)_\gamma, \quad \gamma\in \Gamma.
\]
Note that $\gr(H)$ is a commutative algebra, see \cite[Proposition 1.6]{AgSo0} or \cite[Proposition 6.4]{Zh}. Actually, it is an immediate consequence of \cite[Theorem 11.2.5.a]{Sw}). Therefore, the graded dual $$\gr(H)^*:=\bigoplus_{\gamma\in \Gamma} \Hom(\gr(H)_\gamma, k)$$ is naturally a cocommutative connected $\Gamma$-graded Hopf algebra.

Recall that a $\Gamma$-graded vector space $A=\bigoplus_{\gamma\in \Gamma}A_\gamma$ is called \emph{locally finite} if $A_\gamma$ is finite dimensional for each $\gamma\in \Gamma$. In this case, the \emph{Hilbert series} of $A$ is defined to be the formal series 
\[
\hh_A(\textbf{\rm t}) := \sum_{\gamma\in \Gamma} \dim(A_\gamma)~ \textbf{\rm t}^\gamma \in \mathbb{Z}[[\Gamma]].
\]
Here,  $\mathbb{Z}[[\Gamma]]$ is the ring of formal series on $\Gamma$ with integral coefficients. Its elements are of the form $\sum_{\gamma\in \Gamma} n_\gamma~ \textbf{\rm t}^\gamma$ with $n_\gamma\in \mathbb{Z}$, where $\textbf{\rm t}^\gamma$ is simply a place symbol. 

The following result should be well-known. We provide a proof for reader's convenience. 

\begin{lemma}\label{Hilbert series}
Assume the base field $k$ is of characteristic zero. Let $H$ be a connected $\Gamma$-graded Hopf algebra which is locally finite. Let $P=\bigoplus_{\gamma\in \Gamma} P_\gamma$ be the primitive space of $\gr(H)^*$.  Then 
\[
\hh_H(\textbf{\rm t}) = \hh_{\gr(H)}(\textbf{\rm t}) = \hh_{\gr(H)^*}(\textbf{\rm t}) = \prod_{\gamma\in \Gamma\backslash\{0\}} (1 -\textbf{\rm t}^\gamma)^{-\dim (P_\gamma)}.
\]
\end{lemma}

\begin{proof}
The first two equalities are clear. Since the $\Gamma$-graded Hopf algebra $\gr(H)^*$ is cocommutative, it is isomorphic to the universal enveloping algebra of  $P$ as $\Gamma$-graded Hopf algebras by the Cartier-Milnor-Moore theorem (see \cite[Corollary 4.18 and Theorem 5.18]{MM} or \cite[Theorem 3.8.1]{Cart}). Choose a homogeneous basis $X$ of $P$ and  fix a total order on it. By the well-known Poincar\'{e}-Birkhoff-Witt theorem, the nondecreasing finite product of the elements of $X$ form a homogeneous basis of  $\gr(H)^*$. The last equality then follows by a standard combinatorial argument. 
\end{proof}

\subsection{Diagonalization of the Adams operators}
Assume the base field $k$ is of characteristic zero. Let $H$ be a connected $\Gamma$-graded Hopf algebra, where $\Gamma$ is a nontrivial free abelian monoid.

For any linear map $f\in \End(H)$ with $f(1_H)=0$ and any linear map $g\in \End(H)$ with $g(1_H)=1_H$, one may define $\exp(f),~ \log(g) \in \End(H)$ by the following series expansions
\begin{eqnarray*}
	\exp(f) &:=& \sum_{r\geq 0}\frac{1}{r!} f^{*r}\\
	\log(g) &:=& \sum_{r=1}^{\infty}\frac{(-1)^{r-1}}{r} (g-\eta_H\circ \varepsilon_H)^{*r}.
\end{eqnarray*}
Since the coradical of $H$ equals to $H_0$, they are both well-defined. For integers $n\geq 0$, let
\[
e^{(n)}:= \frac{1}{n!} \log(\id_H)^{*n}.
\]
They are  graded linear maps of $H$ and 
$$e^{(n)}|_{H_\gamma} = 0 \quad \text{for} \quad  \gamma \in \Gamma~ \text{ and } ~ n > N(\gamma),$$ 
Moreover, it is standard to check that
\begin{eqnarray}\label{Eulerian-expansion}
	\Psi_n= \exp(n\log(\id_H)) = \sum_{r\geq 0} n^r ~ e^{(r)}.
\end{eqnarray}

In the case  where $H$ is commutative or cocommutative, the maps $(e^{(n)})_{n\geq 0}$ form a complete orthogonal system of idempotent operators on $H$. That is
\begin{eqnarray}\label{orthogonal-system}
\id_H=\sum_{n\geq0} e^{(n)}, \quad	e^{(n)}\circ e^{(n)} =e^{(n)} \quad \text{and} \quad e^{(m)} \circ e^{(n)} =  0 ~ \text{ for } ~ m\neq n.
\end{eqnarray}
In literature, the map $e^{(n)}$ is called the $n$-th \emph{Eulerian idempotent} of such $H$. It follows that the restriction maps $e^{(0)}|_{H_\gamma}, \ldots, e^{(N(\gamma))}|_{H_\gamma}$ form a complete orthogonal system of idempotent operators on $H_\gamma$, and hence the maps $(e^{(n)}|_{H_\gamma})_{n\geq 0}$ are simultaneous diagonalizable for each $\gamma\in \Gamma$. 

We refer to \cite{Pat} or \cite[Section 4.5]{Lod} for proofs of the above-mentioned results. Apply the expression (\ref{Eulerian-expansion}), one gets the following  result, which is well-known to experts.

\begin{proposition} \label{diagonalizable}
Assume the base field $k$ is of characteristic zero. Let $H=\bigoplus_{\gamma\in \Gamma}H_\gamma$ be a connected $\Gamma$-graded Hopf algebra which is commutative or cocommutative. Then the restrictions $(\Psi_n|_{H_\gamma})_{n\in \mathbb{Z}}$ of the Adams operators are simultaneously diagonalizable for each $\gamma\in \Gamma$. \hfill $\Box$
\end{proposition}

\begin{remark}
The above result only tells us the existence of a common eigenbasis for the Adams operators. In \cite[Section 3]{DPRa} and \cite[Section 2.5]{Pang}, there is presented several algorithms to construct a common eigenbasis for the Adams operators when $H$ is either commutative or cocommutative, by using the first Eulerian idempotent and a combinatorial tool called Lyndon words.
\end{remark}

Proposition \ref{diagonalizable} fails without the assumption that being commutative or cocommutative. In fact, even $\Psi_2|_{H_\gamma}$ itself may be not  diagonalizable as shown by the following example.

\begin{example}\label{counterexample}
Let $H=\mathfrak{S}{\rm Sym}$ be the Hopf algebra of permutations (over the field of rational numbers) introduced by Malvenuto and Reutenauer \cite{MaRe}. It is a connected $\mathbb{N}$-graded Hopf algebra which is neither commutative nor cocommutative. For each $m\in \mathbb{N}$, there are two canonical bases on the the homogeneous component $H_m$. One is the fundamental basis $\{ \mathcal{F}_\sigma ~|~ \sigma \in \mathfrak{S}_m \}$ and the other is the monomial basis $\{ \mathcal{M}_\sigma ~|~ \sigma \in \mathfrak{S}_m \}$, where $\mathfrak{S}_m$ is the $m$-th symmetric group. We refer to \cite{AgSo} for the formulas of the (co)product of $H$ on these bases elements. We us one-line notation for permutations, writing $\sigma=[u_1,\ldots, u_m]$ or simply $\sigma=u_1\ldots u_m$, where $u_i=\sigma(i)$. For $m=3$ it is easy to see
\begin{eqnarray*}
	\Psi_2(\mathcal{F}_{123} ) &=& 4 ~ \mathcal{F}_{123} + \mathcal{F}_{132} + \mathcal{F}_{213} +\mathcal{F}_{231} + \mathcal{F}_{312}\\
	\Psi_2(\mathcal{F}_{132} ) &=& \mathcal{F}_{123} + 4 ~ \mathcal{F}_{132} + 2~ \mathcal{F}_{312} + \mathcal{F}_{321} \\
	\Psi_2(\mathcal{F}_{213} ) &=& \mathcal{F}_{123} + 4 ~ \mathcal{F}_{213} + 2~ \mathcal{F}_{231} + \mathcal{F}_{321} \\
	\Psi_2(\mathcal{F}_{231} ) &=& \mathcal{F}_{123} + 2~\mathcal{F}_{132} + 2 ~ \mathcal{F}_{231} + 2~\mathcal{F}_{312}+ \mathcal{F}_{321} \\
	\Psi_2(\mathcal{F}_{312} ) &=& \mathcal{F}_{123} + 2~\mathcal{F}_{213} + 2~  \mathcal{F}_{231} + 2~ \mathcal{F}_{312} + \mathcal{F}_{321}\\
	\Psi_2(\mathcal{F}_{321} ) &=& \mathcal{F}_{132} + \mathcal{F}_{213} + \mathcal{F}_{231} +  \mathcal{F}_{312} + 4~ \mathcal{F}_{321}
\end{eqnarray*}
and 
\begin{eqnarray*}
	\quad \quad    	\Psi_2(\mathcal{M}_{123} ) &=& 2 ~ \mathcal{M}_{123} \\
	\quad \quad    	\Psi_2(\mathcal{M}_{132} ) &=& 2 ~ \mathcal{M}_{132} \\
	\quad \quad   	\Psi_2(\mathcal{M}_{213} ) &=& 2 ~ \mathcal{M}_{213} \\
	\quad \quad    	\Psi_2(\mathcal{M}_{231} ) &=& \mathcal{M}_{123} + 2~\mathcal{M}_{132} + 3 ~ \mathcal{M}_{231} + \mathcal{M}_{312} \\
	\quad \quad   	\Psi_2(\mathcal{M}_{312} ) &=& \mathcal{M}_{123} + 2~\mathcal{M}_{213} +  \mathcal{M}_{231} + 3~ \mathcal{M}_{312} \\
	\quad   \quad	\Psi_2(\mathcal{M}_{321} ) &=& \mathcal{M}_{132} + \mathcal{M}_{213} + 2~ \mathcal{M}_{231} + 2~ \mathcal{M}_{312} + 8~ \mathcal{M}_{321}.
\end{eqnarray*}
So the characteristic polynomial (resp. minimal polynomial) of $\Psi_2|_{H_3}$ is
\[
(x-2)^4(x-4)(x-8) \quad \Big(~\text{resp. } (x-2)^2(x-4)(x-8) ~ \Big).
\]
Therefore, $\Psi_2|_{H_3}$ is not diagonalizable. In addition, it is worth to mention that the matrix of $\Psi_2|_{H_3}$ with respect to any one of the two bases is not upper triangular for any total order on $\mathfrak{S}_3$.
\end{example}

\section{Main results}
\label{main-results}

We have shown that the Adams operators on connected graded Hopf algebras may not be diagonalizable in general. So it is natural to ask that what kind of weaker linear properties can these maps share.  This section is devoted to present an answer to this question.  

We begin by establishing some notation and conventions. Throughout, let $\Gamma$ stand for a nontrivial free abelian monoid. Suppose that there is given  a connected $\Gamma$-graded algebra $H=\bigoplus_{\gamma\in \Gamma}H_\gamma$ and a family  $\{z_\xi\}_{\xi\in \Xi}$ of homogeneous elements of $H$. We write $\mathcal{P}(\Xi)$ for the set of finite sequences of elements in $\Xi$. For each sequence $V=(\nu_1,\ldots, \nu_r)\in \mathcal{P}(\Xi)$, let 
$$z_V:= z_{\nu_1} \cdots z_{\nu_r}, \quad  l(V):=r \quad \text{ and } \quad |V|:=\deg(z_V)=\sum_{i=1}^r\deg(z_{\nu_i}).$$
By convention, $\mathcal{P}(\Xi)$ contains the empty sequence $\emptyset$ and  we write $z_\emptyset :=1$.

For a partial order $\vartriangleleft$ on $\Xi$ and  a map $h:\Xi\to \mathbb{N}_\infty$, we let
$$\mathcal{P}(\Xi,\vartriangleleft,h):= \{~ ( {\small \overbrace{\xi_{1},\ldots, \xi_{1}}^{p_{1}}}, \ldots,  {\small \overbrace{\xi_{l},\ldots, \xi_{l}}^{p_{l}}})~ |~ \xi_1\vartriangleleft \cdots \vartriangleleft \xi_l\in \Xi, \, p_i <h(\xi_i), \, i=1,\ldots, l ~ \}.$$
In addition, for each $\xi\in \Xi$ let
\begin{eqnarray*}
	H[\xi, \vartriangleleft,h] := {\rm Span} \big(\{~ z_V ~|~V=(\nu_1,\ldots, \nu_m) \in \mathcal{P}(\Xi,\vartriangleleft,h), ~  \nu_m \vartriangleleft \xi  ~\} \big). 
\end{eqnarray*}
Note that $H[\xi, \vartriangleleft,h]$ is a  homogeneous subspace of $H$ that contains $z_\emptyset$. 

For a partial order $\vartriangleleft$ on $\Xi$, we may extend it to two partial orders $\vartriangleleft_L$ and $\vartriangleleft_R$ on $\mathcal{P}(\Xi)$  as follows. For finite sequences $U=(\mu_1,\ldots, \mu_r)$ and $V= (\nu_1,\ldots, \nu_s)$ in $\mathcal{P}(\Xi)$, we write $$U \vartriangleleft_L V \quad  (\text{resp. } ~ U\vartriangleleft_R V) $$
if $|U|+\gamma = |V|$ for some $\gamma\neq 0$, or $|U| = |V|$ but there is an integer $l\leq \min\{r, s\}$ such that $\mu_i =\nu_i$ for $i< l$ and $\mu_l\vartriangleleft \nu_l$ (resp. $\mu_{r+1-i} =\nu_{s+1-i}$ for $i< l$ and $\mu_{r+1-l}\vartriangleleft \nu_{s+1-l}$). Clearly, $\vartriangleleft_L$ and $\vartriangleleft_R$  are both compatible with the left and right concatenation of finite sequences. Moreover, if $\vartriangleleft$ is a total order then any two finite sequences of the same degree are comparable with respect to $\vartriangleleft_L$ and $\vartriangleleft_R$. In the special case that $\Gamma=\mathbb{N}$,  if $\vartriangleleft$ is a total order then so are $\vartriangleleft_L$ and $\vartriangleleft_R$.

Now we are ready to present our main results. The first one characterizes the value of the Adams operators on specific elements of connected graded Hopf algebras.

\begin{theorem}\label{Adams operators}
Let $H=\bigoplus_{\gamma\in \Gamma}H_\gamma$ be a connected $\Gamma$-graded Hopf algebra. Let $\{z_\xi\}_{\xi\in \Xi}$ be a family of homogeneous elements of $H_+$, $\vartriangleleft$ a total order on $\Xi$, and $h:\Xi\to \mathbb{N}_\infty$ a map with $h(\xi)\geq 2$ for each $\xi\in \Xi$. Assume that the following conditions hold:
	\begin{enumerate} 
		\item $\Delta_H(z_\xi) \in 1\otimes z_\xi +z_\xi\otimes 1 +H[\xi, \vartriangleleft,h]_+\otimes H[\xi, \vartriangleleft,h]_+$ for each $\xi\in \Xi$;
		\item $z_\xi^{h(\xi)} \in H[\xi, \vartriangleleft,h]$ for each $\xi\in \Xi$ with $h(\xi) <\infty$;
		\item $z_\nu z_\mu \in k^{\times}\cdot z_\mu z_\nu + H[\nu, \vartriangleleft,h]$ for each pair $\mu, \nu\in \Xi$ with $\mu \vartriangleleft \nu$;
		\item the restriction of $\vartriangleleft$ to the set $\{~ \xi\in \Xi~|~ \deg(z_\xi)=\gamma ~\}$ is a well order for each $\gamma\in \Gamma$.
	\end{enumerate} 
	Then for each integer $n$ and each sequence $V\in \mathcal{P}(\Xi,\vartriangleleft,h)$, one has
	\[ 
	\Psi_n(z_V)\in n^{l(V)} \cdot z_V + {\rm Span} \big(\{~ z_W~|~W  \in \mathcal{P}(\Xi,\vartriangleleft,h), ~ W \vartriangleleft_R V ~\}\big ).
	\]
\end{theorem}

\begin{proof}
In order to be make the reading easier, the proof of this theorem will be addressed in Section \ref{Proof-Adams-operators}. Some auxiliary technical formulas are established in Section \ref{Preparation-Adams-operators}.
\end{proof}

Under the setting of the above theorem, if in addition $\{z_V\}_{V\in \mathcal{P}(\Xi,\vartriangleleft,h)}$ happens to be a basis of $H$, then the result tells us that the restrictions of the Adams operators to each component $H_\gamma$ are simultaneously upper triangularizable with respect to a totally ordered basis, with diagonal entries determined. Below, we present two examples of such setting with $\{z_V\}_{V\in \mathcal{P}(\Xi,\vartriangleleft,h)}$ forms a basis of $H$.

\begin{example}
Let $H$ be a commutative connected $\Gamma$-graded Hopf algebra. Assume that the base field $k$ is of characteristic $0$. Fix a linearly independent family $\{z_\xi\}_{\xi\in \Xi}$ of homogeneous elements of $H_+$ that spans a complement of $(H_+)^2$ in $H_+$. Let $h:\Xi\to \mathbb{N}_\infty$ be the map with constant value $\infty$. By the so-called Leray's theorem (see \cite[Theorem 7.5]{MM} or \cite[Theorem 3.8.3]{Cart}), the family $\{z_V\}_{V\in \mathcal{P}(\Xi,<,h)}$ forms a basis of $H$ for any total order $<$ on $\Xi$. Next we proceed to define a specific total order $\vartriangleleft$ on $\Xi$.  Choose a well order $<_\Gamma$ on $\Gamma$ which is compatible with the addition of $\Gamma$ in the sense that if $\gamma_1<_\Gamma \gamma_2$ then $\gamma+\gamma_1<_\Gamma\gamma+\gamma_2$ for any $\gamma_1,\gamma_2, \gamma \in \Gamma$; also choose  a well order $<_\gamma$ on $\{~\xi\in \Xi~|~\deg(z_\xi)=\gamma~\}$ for each  $\gamma\in \Gamma \backslash\{ 0\}$. Now define a total order $\vartriangleleft$ on $\Xi$  as follows: for $\mu, \nu\in \Xi$,
\begin{eqnarray*}\label{definition-deglex-1}
	\mu \vartriangleleft \nu \,\, \Longleftrightarrow \,\, \left\{
	\begin{array}{llll}
		\deg(z_\mu)<_\Gamma\deg(z_\nu), \quad \text{or} && \\
		\deg(z_\mu) =\deg(z_\nu)\, \text{ and } \mu <_{\deg(z_\mu)} \nu.
	\end{array}\right.
\end{eqnarray*}
It is easy to check that the triple $(\{z_\xi\}_{\xi\in \Xi}, \vartriangleleft, h)$ fulfills the requirements of Theorem \ref{Adams operators}. 
\end{example}

\begin{example}\label{cocommutative}
Let $H$ be a cocommutative connected $\Gamma$-graded Hopf algebra. Assume  $k$ is of characteristic $0$. By the Cartier-Milnor-Moore theorem, $H$ is isomorphic to the universal enveloping algebra of its primitive space $P(H)$ as $\Gamma$-graded Hopf algebras. Fix a homogeneous basis $\{z_\xi\}_{\xi\in \Xi}$ of  $P(H)$. Let $h:\Xi\to \mathbb{N}_\infty$ be the map with constant value $\infty$. By the well-known Poincar\'{e}-Birkhoff-Witt theorem, the family $\{z_V\}_{V\in \mathcal{P}(\Xi,<,h)}$ forms a basis of $H$ for any total order $<$ on $\Xi$. Next we proceed to define a specific total order $\vartriangleleft$ on $\Xi$. As done in the previous example, choose a well order $<_\Gamma$ on $\Gamma$ which is compatible with the addition of $\Gamma$ as well as a well order $<_\gamma$ on $\{~\xi\in \Xi~|~\deg(z_\xi)=\gamma~\}$ for each  $\gamma\in \Gamma \backslash\{ 0\}$. Now define a total order $\vartriangleleft$ on $\Xi$  as follows: for $\mu,\nu\in \Xi$,
\begin{eqnarray*}\label{definition-deglex-2}
	\mu \vartriangleleft \nu \,\, \Longleftrightarrow \,\, \left\{
	\begin{array}{llll}
		\deg(z_\nu)<_\Gamma\deg(z_\mu), \quad \text{or} && \\
		\deg(z_\mu) =\deg(z_\nu)\, \text{ and } \mu <_{\deg(z_\mu)} \nu.
	\end{array}\right.
\end{eqnarray*}
Note that it is different from that of the previous example, where we first require $\deg(z_\mu)<_\Gamma\deg(z_\nu)$. Clearly, the triple $(\{z_\xi\}_{\xi\in \Xi}, \vartriangleleft, h)$ fulfills the requirements of Theorem \ref{Adams operators} except the third one. But for any pair $\mu, \nu\in \Xi$ with  $\mu \vartriangleleft \nu$, one has $[z_\nu, z_\mu]\in P(H)$ and therefore \[z_\nu z_\mu -z_\mu z_\nu \in {\rm Span} \big(\{~ z_\xi~|~\xi\in \Xi, ~ ~ \deg(z_\xi) =\deg (z_\mu) +\deg(z_v) ~\}.  \]
The third condition of Theorem \ref{Adams operators} then follows by the very definition of $\vartriangleleft$.  
\end{example}

One may ask that  does such kind of  triples $(\{z_\xi\}_{\xi\in \Xi}, \vartriangleleft, h)$ exist for general connected graded Hopf algebras? The following result gives a positive answer.


\begin{theorem}\label{PBW-basis}
Let $H=\bigoplus_{\gamma\in \Gamma}H_\gamma$ be a connected $\Gamma$-graded Hopf algebra. Then there exists a family $\{z_\xi\}_{\xi\in \Xi}$  of homogeneous elements of $H_+$,  a total order $\vartriangleleft$ on $\Xi$, and a map  $h:\Xi\to \mathbb{N}_\infty$ with $h(\xi)\geq 2$ for each $\xi\in \Xi$ such that the  following conditions hold:
\begin{enumerate}
	\item $\{z_V\}_{V\in \mathcal{P}(\Xi,\vartriangleleft,h)}$ is a basis of $H$;
	\item $\Delta_H(z_\xi) \in 1\otimes z_\xi +z_\xi\otimes 1 +H[\xi, \vartriangleleft,h]_+\otimes H[\xi, \vartriangleleft,h]_+$ for each $\xi\in \Xi$;
	\item $z_\xi^{h(\xi)} \in H[\xi, \vartriangleleft,h]$ for each $\xi\in \Xi$ with $h(\xi) <\infty$;
	\item $z_\nu z_\mu \in z_\mu z_\nu + H[\nu, \vartriangleleft,h]$ for each pair $\mu, \nu\in \Xi$ with $\mu \vartriangleleft \nu$;
	\item if $k$ is of characteristic $0$ then $h(\xi)$ is infinity for each $\xi\in \Xi$; while if $k$ is of characteristic $p>0$ then $h(\xi)$ is either infinity or a power of $p$ for  each $\xi\in \Xi$;
	\item the restriction of $\vartriangleleft$ to the set $\{~ \xi\in \Xi~|~ \deg(z_\xi)=\gamma ~\}$ is a well order for each $\gamma\in \Gamma$.
\end{enumerate}
\end{theorem}

\begin{proof}
We postpone the proof to Section \ref{proof-PBW}. Actually, without Parts (2) and (6) it is simply a special case of \cite[Theorem 4.3]{Zhou}. Also, see \cite[Theorem A]{ZSL3} and \cite[Theorem A]{LZ} for the characteristic zero case. However, we will show briefly in Section \ref{proof-PBW} how the desired data are constructed and why they are workable. It is mainly for self-containedness and reader's convenience.
\end{proof}

In the remainder of this section, we show some linear properties that the Adams operators on connected graded Hopf algebras share as consequences of the above two theorems. Some of them have been obtained in \cite{AgLa} by a different method when the base field $k$ is of characteristic zero.

\begin{corollary}
Let $H=\bigoplus_{\gamma\in \Gamma}H_\gamma$ be a connected $\Gamma$-graded Hopf algebra. Then
\begin{enumerate}
	\item $(\Psi_n|_{H_\gamma})_{n\in \mathbb{Z}}$ are simultaneously upper triangularizable for each $\gamma\in \Gamma$.
	\item the eigenvalues of $\Psi_n|_{H_\gamma}$ are powers of $n$ (in $k$) for each   $n\in \mathbb{Z}$ and $\gamma\in \Gamma$.
	\item $S_H|_{H_\gamma}$ is diagonalizable for each $\gamma\in \Gamma$ if and only if $(S_H)^2=\id_H$.
	\item  $\Psi_n$ is composition invertible for each $n\in\mathbb{Z}$ which is not zero in $k$.
\end{enumerate}
\end{corollary}

\begin{proof}
It is a direct consequence of Theorem \ref{Adams operators} and Theorem \ref{PBW-basis}.
\end{proof}

Let $H$ be a locally finite connected $\Gamma$-graded Hopf algebra. For each $n\in \mathbb{N}$ and $\gamma\in \Gamma$, let 
\begin{eqnarray*}
	\text{\rm mul}_H(n,\gamma):= \sum_{(d_\alpha)_{\alpha \in \Gamma\backslash\{0\}}} \prod_{\beta : d_\beta\neq 0}\binom{\dim (P_\beta)+d_\beta-1}{d_\beta },
\end{eqnarray*}
where $P=\bigoplus_{\beta\in \Gamma} P_\beta$ is the primitive space of $\gr(H)^*$, and  $(d_\alpha)_{\alpha\in \Gamma\backslash\{0\}}$ runs over all families of numbers with only finitely many nonzero entries such that $$\sum_{\alpha\in \Gamma\backslash\{0\}} d_\alpha =n \quad \text{and} \quad \sum_{\alpha\in \Gamma\backslash\{0\}} d_\alpha ~ \alpha =\gamma.$$ 
Note (Lemma \ref{Hilbert series}) that the  numbers $\text{\rm mul}_H(n,\gamma)$ are determined by the Hilbert series $\hh_H(\textbf{\rm t})$ and
$$\text{\rm mul}_H(n,\gamma)=0 \quad \text{for} \quad n>N(\gamma).$$
We refer to \cite[Subsections 2.1 and 2.2]{AgLa} for discussion and more detail. Recall that $N(\gamma)$ denotes the largest number  such that $\gamma$ can be expressed as $\gamma= \alpha_1+\cdots +\alpha_{N(\gamma)}$ with $\alpha_i\in \Gamma \backslash\{0\}$. 

\begin{lemma}\label{counting basis element}
Assume the base field $k$ is of characteristic zero. Let $H$ be a connected $\Gamma$-graded Hopf algebra which is locally finite. Let $(\{z_\xi\}_{\xi\in \Xi}, \vartriangleleft, h)$ be a triple as in Theorem \ref{Adams operators}. Assume further that $h$ always takes the constant value $\infty$ and $\{z_V\}_{V\in \mathcal{P}(\Xi,\vartriangleleft,h)}$ is a basis of $H$. Then 
\[
\text{\rm mul}_H(n,\gamma)= \#\{~ V\in \mathcal{P}(\Xi,\vartriangleleft,h) ~| ~ l(V)=n, ~|V|=\gamma ~\}.
\] 
\end{lemma}

\begin{proof}
For each $\gamma\in \Gamma$, let $p_\gamma$ be the number of $\xi\in \Xi$ such that $\deg(z_\xi)=\gamma$. Then by a standard combinatorial argument, one has 
\[
\hh_H(\textbf{\rm t}) = \prod_{\gamma\in \Gamma\backslash\{0\}} (1 -\textbf{\rm t}^\gamma)^{-p_\gamma}
\] 
and  
\[
\#\{~ V\in \mathcal{P}(\Xi,\vartriangleleft,h) ~| ~ l(V)=n, ~|V|=\gamma ~\}= \sum_{(d_\alpha)_{\alpha \in \Gamma\backslash\{0\}}} \prod_{\beta : d_\beta\neq 0}\binom{p_\beta + d_\beta-1}{d_\beta }
\]
for each $n\in \mathbb{N}$ and $\gamma\in \Gamma$, where $(d_\alpha)_{\alpha\in \Gamma\backslash\{0\}}$ runs over all families of numbers with only finitely many nonzero entries such that $\sum_{\alpha\in \Gamma\backslash\{0\}} d_\alpha =n$ and $\sum_{\alpha\in \Gamma\backslash\{0\}} d_\alpha ~ \alpha =\gamma$. By  Lemma \ref{Hilbert series}, one obtains $p_\beta=\dim(P_\beta)$ for each $\beta\in \Gamma$. The desired equality follows immediately.
\end{proof}

\begin{corollary}\label{characteristic polynomial}
Assume that the base field $k$ is of characteristic zero. Let $H$ be a connected $\Gamma$-graded Hopf algebra which is locally finite. Then for each $n\in \mathbb{Z}$ and $\gamma\in \Gamma$, the characteristic polynomial  of the restriction $\Psi_n|_{H_\gamma}$ of the $n$-th Adams operator to $H_\gamma$ is
\[
\mathcal{X}(\Psi_n|_{H_\gamma}) (x) =\prod_{s=0}^{N(\gamma)} (x-n^s)^{\text{\rm mul}_H(s,\gamma)}.
\]
\end{corollary}

\begin{proof}
It is a direct consequence of Theorem \ref{Adams operators}, Theorem \ref{PBW-basis} and Lemma \ref{counting basis element}.
\end{proof}

\begin{remark}
The above corollary recovers \cite[Theorem 3]{AgLa}. It is the main result of \cite{AgLa} and the approach there is different from ours. We refer to loc. cit. for many interesting consequences of this result and for computations as well as concrete descriptions on various examples.
\end{remark}

\section{Preparations for the proof of Theorem \ref{Adams operators}}
\label{Preparation-Adams-operators}

In this section, we make some preparations to prove Theorem \ref{Adams operators}, which is one of the main results of this paper. Among others, some auxiliary technical  formulas are established. 

Throughout this section, we fix a  connected $\Gamma$-graded algebra $H=\bigoplus_{\gamma\in \Gamma} H_\gamma$, where $\Gamma$ is a nontrivial free abelian monoid. We also fix a family $\{z_\xi\}_{\xi\in \Xi}$ of homogeneous elements of $H_+$, a total order $\vartriangleleft$ on $\Xi$ and a map $h:\Xi\to \mathbb{N}_\infty$  with $h(\xi)\geq 2$ for each $\xi\in \Xi$. 

In the sequel, the notation and conventions of the previous section are retained.

Recall that a partial order $<$ on a set $A$ is said to be \emph{well-founded} if  every nonempty subset of $A$ has a minimal element with respect to $<$, or equivalently, there is no infinitely decreasing sequence with respect to $<$. Note that a well order on $A$ is simply a well-founded total order on $A$.

\begin{lemma}\label{well-founded-order}
Assume that the restriction of $\vartriangleleft$ to the set $\{~ \xi\in \Xi~|~ \deg(z_\xi)=\gamma ~\}$ is well-founded for each $\gamma\in \Gamma$. Then the partial orders $\vartriangleleft_L$ and $\vartriangleleft_R$ on $\mathcal{P}(\Xi)$ are both well-founded.
\end{lemma}

\begin{proof}
We show $\vartriangleleft_L$ is well-founded. The proof for that of $\vartriangleleft_R$ is similar. Firstly, we introduce a partial order $<$ on $\Gamma$ as follows: for $\alpha, \beta\in \Gamma$, we write  $\alpha <\beta$ iff there exists a nonzero element $\gamma \in \Gamma$ such that $\alpha+ \gamma =\beta$. Since $\Gamma$ is a free abelian monoid, it is a well-founded partial order on $\Gamma$. To see the result, it suffices to see every nonincreasing sequence  $V_1 \trianglerighteq_L V_2 \trianglerighteq_L \cdots$ of elements in $\mathcal{P}(\Xi)$ of  degree  bounded by some $\gamma$ stabilizes eventually. We prove it by induction on $\gamma$ with respect to $<$.
	
For the case that $\gamma = 0$, the desired result clearly holds. Now we assume $\gamma >0$. Let $\xi_i\in \Xi$ be the first component of $V_i$, and let  $W_i$ be the remaining of $V_i$ by eliminating the first component. Note that there are only finitely many elements of $\Gamma$ that $\leq \gamma$, one has $$|V_p| =|V_{p+1}| =\cdots$$ and hence $\xi_p\trianglerighteq \xi_{p+1}\trianglerighteq \cdots$ for some $p\geq 1$. 
The assumption of $\vartriangleleft$ implies that  $\xi_q=\xi_{q+1}=\cdots$ and hence $W_q \trianglerighteq_L W_{q+1} \trianglerighteq_L \cdots$  for some $q\geq p$. But $W_q, W_{q+1},\ldots$ have  degrees bounded by $\gamma- \deg(z_{\xi_q})$, which is less than $\gamma$ with respect to $<$, the result follows by the induction hypothesis.	
\end{proof}

\begin{lemma}\label{rearragement}
Assume that the following conditions hold:
	\begin{enumerate}
		\item $z_\xi^{h(\xi)} \in H[\xi, \vartriangleleft,h]$ for each $\xi\in \Xi$ with $h(\xi) <\infty$,
		\item $z_\nu z_\mu \in a_{\nu,\mu}\cdot z_\mu z_\nu + H[\nu, \vartriangleleft,h]$ for each pair $\mu, \nu\in \Xi$ with $\mu \vartriangleleft \nu$, where $a_{\nu,\mu}\in k^{\times}$.
		\item the restriction of $\vartriangleleft$ to the set $\{~ \xi\in \Xi~|~ \deg(z_\xi)=\gamma ~\}$ is a well order for each $\gamma\in \Gamma$.
	\end{enumerate} 
	Then for each element  $V=(\nu_1,\ldots, \nu_s)\in \mathcal{P}(\Xi)$, one has
	\begin{eqnarray*}
		z_V \in 
		\left\{
		\begin{array}{ll} 
			~ {\rm Span} \big(\{~ z_W~|~W  \in \mathcal{P}(\Xi,\vartriangleleft,h), ~ W \vartriangleleft_R \Pi(V) ~\} \big), & \text{\rm if } ~ ~  \Pi(V)  \not\in \mathcal{P}(\Xi,\vartriangleleft,h),\\
			~  (\prod_{1\leq i<j\leq s \atop \nu_{i} \vartriangleright \nu_j} a_{\nu_i,\nu_j}) \cdot  z_{\Pi(V)} + {\rm Span} \big(\{~ z_W~|~W  \in \mathcal{P}(\Xi,\vartriangleleft,h), ~ W \vartriangleleft_R \Pi(V) ~\}\big ), & \text{\rm if } ~ ~ \Pi(V)  \in \mathcal{P}(\Xi,\vartriangleleft,h), 
		\end{array}
		\right.
	\end{eqnarray*}
	where $\Pi(V)=(\omega_1,\ldots, \omega_s)$ is the unique permutation of $V$ such that  $\omega_1 \trianglelefteq \cdots \trianglelefteq \omega_s$. 
\end{lemma}

\begin{proof}
We proceed by induction on $V=(\nu_1,\ldots, \nu_s)$ with respect to  $\vartriangleleft_L$. If $V$ is minimal, that is $V$ is the empty sequence, then the result holds by convention. Next we assume $V$ is not minimal.
	
First consider the case that $V$ is  nondecreasing with respect to $\vartriangleleft$. If $V\in \mathcal{P}(\Xi,\vartriangleleft,h)$, then the desired formula clearly holds. Otherwise, suppose $V\not \in \mathcal{P}(\Xi,\vartriangleleft,h)$ and write it in the form $$V= (\overbrace{\xi_{1},\ldots, \xi_{1}}^{p_{1}}, \ldots,  \overbrace{\xi_{r},\ldots, \xi_{r}}^{p_{r}}), \quad p_i\geq 1, \ \  \xi_1 \vartriangleleft  \cdots \vartriangleleft \xi_r.$$  By the assumption, one may fix an integer $i$ such that $p_i\geq h(\xi_i)$, and then one has that 
	\begin{align}
		z_V  \ \  = \ \  \sum_{U}a_U\cdot z_{\xi_1}^{p_1} \cdots z_{\xi_{i-1}}^{p_{i-1}} (z_U) z_{\xi_i}^{p_i - h(\xi_i)} z_{\xi_{i+1}}^{p_{i+1}} \cdots z_{\xi_r}^{p_r} \tag{\dag} \label{dag}
	\end{align}
	for some $a_U\in k$, where $U=(\mu_1,\ldots, \mu_t)$ runs over elements in $\mathcal{P}(\Xi,\vartriangleleft,h)$ with  $\mu_1,\ldots, \mu_t \vartriangleleft \xi_i$ and  $|U| = \deg(z_{\xi_i})\cdot h(\xi_i)$. It is easy to check that the finite sequences
	\[
	( {\small \overbrace{\xi_{1},\ldots, \xi_{1}}^{p_{1}}}, \ldots, {\small \overbrace{\xi_{i-1},\ldots, \xi_{i-1}}^{p_{i-1}}}, \mu_1,\ldots, \mu_t,  {\small \overbrace{\xi_{i},\ldots, \xi_{i}}^{p_i-h(\xi_i)}},  {\small \overbrace{\xi_{i+1},\ldots, \xi_{i+1}}^{p_{i+1}}},\ldots, {\small \overbrace{\xi_{r},\ldots, \xi_{r}}^{p_{r}}}) ~ \in ~ \mathcal{P}(\Xi)
	\]
	appearing in (\ref{dag}) are all smaller than $V$ with respect to $\vartriangleleft_L$ and the nondecreasing rearrangement of them  are all smaller than $\Pi(V)$ with respect to $\vartriangleleft_R$. The desired formula follows by the induction hypothesis.
	
	Now  consider the case that $V$ is not nondecreasing. Assume $\nu_{i+1} \vartriangleleft \nu_i$ for some $i$. Then 
	\begin{align}
		z_V \ \  = \ \  a_{\nu_{i},\nu_{i+1}}\cdot \big(z_{\nu_1}\cdots z_{\nu_{i-1}} (z_{\nu_{i+1}} z_{\nu_{i}}) z_{\nu_{i+2}}\cdots z_{\nu_{s}}\big)
		+ \sum_{U} a_U \cdot z_{\nu_1}\cdots z_{\nu_{i-1}} (z_U) z_{\nu_{i+2}}\cdots z_{\nu_{s}}, \tag{\ddag} \label{ddag}
	\end{align}
	for some $a_U\in k$, where $U=(\mu_1,\ldots, \mu_t)$ runs over elements in $\mathcal{P}(\Xi,\vartriangleleft,h)$ with  $\mu_1,\ldots, \mu_t \vartriangleleft \nu_i$ and  $|U| = \deg(z_{\nu_i})+\deg(z_{\nu_{i+1}})$.  It is easy to check that the finite sequences
	\begin{eqnarray*}
		(\nu_1,\ldots, \nu_{i-1}, \nu_{i+1}, \nu_i, \nu_{i+2},\ldots, \nu_s), \quad  	(\nu_1,\ldots, \nu_{i-1}, \mu_1, \ldots, \mu_t, \nu_{i+2},\ldots, \nu_s) ~~~ \in ~~~ \mathcal{P}(\Xi)
	\end{eqnarray*} 
	appearing in (\ref{ddag}) are all smaller than $V=(\nu_1,\ldots, \nu_s)$ with respect to $\vartriangleleft_L$. In addition, $\Pi(V)$ is also the nondecreasing rearrangement of $(\nu_1,\ldots, \nu_{i-1}, \nu_{i+1}, \nu_i, \nu_{i+2},\ldots, \nu_s)$, and it is bigger with respect to $\vartriangleleft_R$ than that of $(\nu_1,\ldots, \nu_{i-1}, \mu_1, \ldots, \mu_t, \nu_{i+2},\ldots, \nu_s)$, the other sequences  appearing in (\ref{ddag}). The desired formula then follows immediately by the induction hypothesis.
\end{proof}

For finite sequences $V, ~ W\in \mathcal{P}(\Xi)$, we write $W|V$ to mean that $W$ is a subsequence of $V$. It means that if $V=(\nu_1,\ldots, \nu_n)$ then $W=(\nu_{i_1},\ldots, \nu_{i_s})$ for some integers  $1\leq i_1<\cdots <i_s\leq n$. If $V= (\overbrace{\xi_{1},\ldots, \xi_{1}}^{p_{1}}, \ldots,  \overbrace{\xi_{r},\ldots, \xi_{r}}^{p_{r}})$ with $\xi_1,\ldots, \xi_r$ mutually distinct, then its subsequences are necessarily of the form $W= (\overbrace{\xi_{1},\ldots, \xi_{1}}^{q_{1}}, \ldots,  \overbrace{\xi_{r},\ldots, \xi_{r}}^{q_{r}})$ with $q_i\leq p_i$. In this case, we write 
\[
\binom{V}{W}:= \binom{p_1}{q_1}\cdots \binom{p_r}{q_r}
\]
and 
\[
V/W:= (\overbrace{\xi_{1},\ldots, \xi_{1}}^{p_1-q_{1}}, \ldots,  \overbrace{\xi_{r},\ldots, \xi_{r}}^{p_r-q_{r}}). 
\]
  
The formula in the next result is the key observation of this paper. 

\begin{lemma}\label{comultiplicatoin-expression}
Assume that the conditions of Lemma \ref{rearragement} hold for a connected $\Gamma$-graded algebra $H$. Let $\Delta:H\to H\otimes H$ be a homomorphism of $\Gamma$-graded algebras such that $$\Delta(z_\xi) \in 1\otimes z_\xi +z_\xi\otimes 1 +H[\xi, \vartriangleleft,h]_+\otimes H[\xi, \vartriangleleft,h]_+, \quad \xi\in \Xi.$$ Then  for each nonempty finite sequence $V\in \mathcal{P}(\Xi,\vartriangleleft,h)$ one has
	$$\Delta(z_V)\in \sum_{W|V} \binom{V}{W} ~ z_W\otimes z_{V/W} + {\rm Span} \big(\{~ z_{M}\otimes z_N~|~ M, N  \in \mathcal{P}(\Xi,\vartriangleleft,h)\backslash\{\emptyset\}, ~ \Pi(MN) \vartriangleleft_R V ~\}\big ), $$
	where $\Pi$ is the rearrange map given in Lemma \ref{rearragement}.
\end{lemma}

\begin{proof}
	Firstly we show that for each $\xi \in \Xi$ and $n\geq 1$ one has
	\[
	\Delta(z_\xi^n) \in \sum_{i=0}^n \binom{n}{i} ~z_\xi^i\otimes z_\xi^{n-i} + \sum\nolimits_{r,s\geq 0 \atop r+s<n} (H[\xi, \vartriangleleft,h]_+\cdot z_\xi^r)\otimes (H[\xi, \vartriangleleft,h]_+\cdot z_\xi^s).
	\]
	We do it by induction on $n$. The initial step is by the assumption of $\Delta$. Then apply the  decomposition $\Delta(z_\xi^n)= \Delta(z_\xi) ~ \Delta(z_\xi^{n-1})$, the induction step is an easy consequence of Lemma \ref{rearragement}. It follows that the desired formula holds for the case that $V=(\xi,\ldots, \xi)$ with $\xi$ appears $n$ times.
	
	For the general case that $V= (\overbrace{\xi_{1},\ldots, \xi_{1}}^{p_{1}}, \ldots,  \overbrace{\xi_{r},\ldots, \xi_{r}}^{p_{r}})$ with $r\geq 1$, $p_i\geq 1$ and  $\xi_1,\ldots, \xi_r$ mutually distinct, we show the formula by induction on $r$. The case that $r=1$ has proved in the above paragraph. For $r\geq 2$, we decompose $V=V'V''$ with $V'= (\overbrace{\xi_{1},\ldots, \xi_{1}}^{p_{1}})$. Then by induction, 
	\begin{eqnarray*}
		\Delta(z_V) &=& \Delta(z_{V'}) \Delta(z_{V''}) \\
		&= & \Big(\sum_{W'|V'} \binom{V'}{W'} ~ z_{W'}\otimes z_{V'/W'} +f' \Big) \cdot \Big (\sum_{W''|V''} \binom{V''}{W''} ~ z_{W''}\otimes z_{V''/W''} +f'' \Big),
	\end{eqnarray*}
	where $f'$  is a linear combination of $z_{M'}\otimes z_{N'}$ with $M', N'  \in \mathcal{P}(\Xi,\vartriangleleft,h)\backslash\{\emptyset\} $ and $\Pi(M'N') \vartriangleleft_R V'$ and likewise for $f''$. Since that $\xi_1 \vartriangleleft \xi_i$ for $i\geq 2$, one has 
	\[
	\Big(\sum_{W'|V'} \binom{V'}{W'} ~ z_{W'}\otimes z_{V'/W'}\Big)~ \Big (\sum_{W''|V''} \binom{V''}{W''} ~ z_{W''}\otimes z_{V''/W''} \Big) = \sum_{W|V} \binom{V}{W} ~ z_W\otimes z_{V/W}.
	\]
	It remains to show the elements $(z_{W'}\otimes z_{V'/W'})\cdot f''$, $f'\cdot (z_{W''}\otimes z_{V''/W''})$ and $f'\cdot f''$ belongs to 
	\[
	 {\rm Span} \big(\{~ z_{M}\otimes z_N~|~ M, N  \in \mathcal{P}(\Xi,\vartriangleleft,h)\backslash\{\emptyset\}, ~ \Pi(MN) \vartriangleleft_R V ~\}\big ).
	\]
	To this end, note that for any elements $A_1,A_2,B_1,B_2\in \mathcal{P}(\Xi)$, if $\Pi(A_1)\trianglelefteq_R \Pi(B_1)$ and $\Pi(A_2)\trianglelefteq_R \Pi(B_2)$ with at most one equality holds then $\Pi(A_1A_2)\vartriangleleft_R \Pi(B_1B_2)$. Since $f'f''$ is a linear combination of these $z_{M'}z_{M''} \otimes z_{N'}z_{N''}$ with $\Pi(M'N') \vartriangleleft_R V'$ and $\Pi(M''N'') \vartriangleleft_R V''$, the desired result follows. Similar discussion applies to the 
   the expressions $(z_{W'}\otimes z_{V'/W'})\cdot f''$ and $f'\cdot (z_{W''}\otimes z_{V''/W''})$.
\end{proof}

\begin{proposition}\label{convolution-expression}
Assume that the conditions of Lemma \ref{rearragement} hold for a connected $\Gamma$-graded algebra $H$. Let $\Delta:H\to H\otimes H$ be a homomorphism of $\Gamma$-graded algebras such that $$\Delta(z_\xi) \in 1\otimes z_\xi +z_\xi\otimes 1 +H[\xi, \vartriangleleft,h]_+\otimes H[\xi, \vartriangleleft,h]_+, \quad \xi\in \Xi.$$  Let $f_1,f_2:H\to H$ be two $\Gamma$-graded linear maps such that    $$f_i(z_V) \in b_i(V)\cdot z_V+{\rm Span} \big(\{~ z_W~|~W  \in \mathcal{P}(\Xi,\vartriangleleft,h), ~ W \vartriangleleft_R V ~\}\big ), \quad i=1,2 $$  with $b_i(V)\in k$ for each $V\in \mathcal{P}(\Xi,\vartriangleleft,h)$. 
	Then the map $f:=\mu_H\circ (f_1\otimes f_2)\circ \Delta$ satisfies 
	$$f(z_V) \in b(V)\cdot z_V+{\rm Span} \big(\{~ z_W~|~W  \in \mathcal{P}(\Xi,\vartriangleleft,h), ~ W \vartriangleleft_R V ~\}\big ), $$
	where $b(V)= \sum_{W|V} \binom{V}{W}~ b_1(W)~ b_2(V/W)$, for each $V\in \mathcal{P}(\Xi,\vartriangleleft,h)$.
\end{proposition}

\begin{proof}
It is an easy consequence of Lemma \ref{rearragement} and Lemma \ref{comultiplicatoin-expression} .
\end{proof}

The next proposition is independent of the main results of this paper. However, it has its own right of interests. We consider it as an organic part of the method used in this section.

\begin{proposition}\label{chang-of-order}
Assume that the conditions of Lemma \ref{rearragement} hold  for a connected $\Gamma$-graded algebra $H$. 
If the family $\{z_V\}_{V\in \mathcal{P}(\Xi,\vartriangleleft,h)}$ of elements of $H$ is linearly independent (resp. spans $H$), then so is the family $\{z_V\}_{V\in \mathcal{P}(\Xi,<,h)}$ of elements of $H$ for any total order $<$ on $\Xi$.
\end{proposition}

\begin{proof}
	Note that the rearrange map $\Pi: \mathcal{P}(\Xi,<,h) \to \mathcal{P}(\Xi,\vartriangleleft,h)$ defined as in Lemma \ref{rearragement} is a  bijection. 
	
	(1) Assume $\{z_V\}_{V\in \mathcal{P}(\Xi,\vartriangleleft,h)}$ is linearly independent but $\{z_V\}_{V\in \mathcal{P}(\Xi,<,h)}$ is not. Then there is a positive integer $p\geq1$, distinct finite sequences $V_1, \ldots, V_p \in \mathcal{P}(\Xi,<,h)$ of the same degree  and nonzero scalars $a_1,\ldots, a_p\in k$ such that 
	$\sum_{i=1}^p a_i\cdot z_{V_i} =0.$
	Without loss of generality, we may assume 
	$$\Pi(V_1)\vartriangleleft_R \Pi(V_2) \vartriangleleft_R  \cdots \vartriangleleft_R \Pi(V_p). $$ Then by Lemma \ref{rearragement}, one has
	\begin{eqnarray*}
		z_{\Pi(V_p)} & \in &  {\rm Span} \big(\{~ z_W~|~W  \in \mathcal{P}(\Xi,\vartriangleleft,h), ~ W \vartriangleleft_R \Pi(V_p) ~\}\big ),
	\end{eqnarray*}
	which contradicts that  $\{z_V\}_{V\in \mathcal{P}(\Xi,\vartriangleleft,h)}$ is linearly independent.
	
	(2) Assume 	$\{z_V\}_{V\in \mathcal{P}(\Xi,\vartriangleleft,h)}$ spans $H$. Let $A\subseteq H$ be the subspace spanned by  $\{z_V\}_{V\in \mathcal{P}(\Xi,<,h)}$. To see $A=H$ it suffices to show  $z_V\in A$ for any $V\in \mathcal{P}(\Xi,\vartriangleleft,h)$. 
	We prove it by induction on $V$ with respect to $\vartriangleleft_R$. If $V$ is minimal, namely $V=\emptyset$, then the result holds by convention. 
	 Otherwise,  
	\begin{eqnarray*}
		z_V &\in& k^{\times} \cdot z_{\Pi^{-1}(V)} +  {\rm Span} \big(\{~ z_W~|~W  \in \mathcal{P}(\Xi,\vartriangleleft,h), ~ W \vartriangleleft_R V ~\}\big )
	\end{eqnarray*} 
	by Lemma \ref{rearragement}. Therefore, $A=H$ by the induction hypothesis.
\end{proof}

\section{Proof of Theorem \ref{Adams operators}}
\label{Proof-Adams-operators}

In this section, we address the proof of Theorem \ref{Adams operators}. We continue to use the notation and conventions employed in the previous sections. By convention, $0^0:=1$. 

\begin{lemma}\label{general formula}
Let $H$ be a connected $\Gamma$-graded Hopf algebra, where $\Gamma$ is a nontrivial free abelian monoid. Let $T:H\to H$ be a graded linear map satisfying that 
$$T(z_V)\in a^{l(V)} \cdot z_V + {\rm Span} \big(\{~ z_W~|~W  \in \mathcal{P}(\Xi,\vartriangleleft,h), ~ W \vartriangleleft_R V ~\}\big )$$ for each $V\in \mathcal{P}(\Xi,\vartriangleleft,h)$, where $a\in k$ is a constant coefficient. Then 
\[T^{*n}(z_V)\in (na)^{l(V)} \cdot z_V + {\rm Span} \big(\{~ z_W~|~W  \in \mathcal{P}(\Xi,\vartriangleleft,h), ~ W \vartriangleleft_R V ~\}\big )\]
for each $V\in \mathcal{P}(\Xi,\vartriangleleft,h)$ and $n\geq 1$.
\end{lemma}

\begin{proof}
We prove it by induction on $n$. The case $n=1$ is by assumption. Now assume $n\geq 2$ and $$T^{*n-1}(z_V)\in ((n-1)a)^{l(V)} \cdot z_V + {\rm Span} \big(\{~ z_W~|~W  \in \mathcal{P}(\Xi,\vartriangleleft,h), ~ W \vartriangleleft_R V ~\}\big ).$$ Consider the decomposition $T^{*n} = T* T^{*n-1}$. Proposition \ref{convolution-expression} tells us that
\[
T^{*n}(z_V)\in b(V) \cdot z_V + {\rm Span} \big(\{~ z_W~|~W  \in \mathcal{P}(\Xi,\vartriangleleft,h), ~ W \vartriangleleft_R V ~\}\big )
\]
with 
\[
b(V) = \sum_{W|V} \binom{V}{W}~ a^{l(W)}~ \big((n-1)a\big)^{l(V/W)}.
\]
It remains to show $b(V)= (na)^{l(V)}$ for the induction step. To this end, let
$V= (\overbrace{\xi_{1},\ldots, \xi_{1}}^{p_{1}}, \ldots,  \overbrace{\xi_{r},\ldots, \xi_{r}}^{p_{r}})$ with $\xi_1,\ldots, \xi_r \in \Xi$ being mutually distinct. It is easy to check that
\begin{eqnarray*}
\sum_{W|V} \binom{V}{W}~ a^{l(W)}~ \big((n-1)a \big)^{l(V/W)} & = & \prod_{i=1}^{r} \Big(\sum_{s_i=0}^{p_i} \binom{p_i}{s_i} ~ a^{s_i} \big((n-1)a\big)^{p_i-s_i} \Big ) \\
&= & \prod_{i=1}^{r} \Big(a+ (n-1)a \Big)^{p_i} \\
&=& (na)^{l(V)}.
\end{eqnarray*}
So $b(V)= (na)^{l(V)}$ as desired. This complete the proof of  the result.
\end{proof}

\begin{proof}[Proof of Theorem \ref{Adams operators}]
The result for $n=0$ is clear, since $\Psi_0:= \eta_H\circ \varepsilon_H$. The result for $n\geq 1$ is a direct consequence of Lemma \ref{general formula} by considering $T=\id_H$ with $a=1$. Note that $\Psi_{n} = (S_H)^{* -n}$. So by Lemma \ref{general formula}, to see the result for $n\leq -1$, it suffices to show 
\[
S_H(z_V) \in (-1)^{l(V)} z_V+ {\rm Span} \big(\{~ z_W~|~W  \in \mathcal{P}(\Xi,\vartriangleleft,h), ~ W \vartriangleleft_R V ~\}\big )
\]
for each $V\in \mathcal{P}(\Xi,\vartriangleleft,h)$.
By Proposition \ref{convolution-expression}, one has
\[
(\eta_H\circ \varepsilon_H -\id_H)^{* n} (z_V) \in k\cdot z_V + {\rm Span} \big(\{~ z_W~|~W  \in \mathcal{P}(\Xi,\vartriangleleft,h), ~ W \vartriangleleft_R V ~\}\big )
\]
for each $V\in \mathcal{P}(\Xi,\vartriangleleft,h)$ and $n\geq 0$. 
It follows from the formula $S_H=\sum_{n\geq 0} (\eta_H\circ \varepsilon_H -\id_H)^{* n}$ that
\[
S_H(z_V) \in b(V)\cdot z_V+ {\rm Span} \big(\{~ z_W~|~W  \in \mathcal{P}(\Xi,\vartriangleleft,h), ~ W \vartriangleleft_R V ~\}\big )
\]
for each $V\in \mathcal{P}(\Xi,\vartriangleleft,h)$, where $b(V)\in k$. It remains to prove $b(V)=(-1)^{l(V)}$. We do it by induction on $l(V)$. The case for $l(V)=0$ is clear, since $z_\emptyset =1$. Now assume $l(V)=n>0$ and the desired formula holds for $l(V)<n$. Apply Proposition \ref{convolution-expression} to the decomposition $\eta_H\circ \varepsilon_H  = \id_H* S_H$, one gets
\begin{eqnarray*}
0 &=& \binom{V}{\emptyset}1^{l(\emptyset)}b(V) -\binom{V}{\emptyset}1^{l(\emptyset)} (-1)^{l(V)} + \sum_{W|V} \binom{V}{W}~ 1^{l(W)}~ (-1)^{l(V/W)} \\
 &=& b(V)-(-1)^{l(V)} + \Big(1+(-1) \Big)^{l(V)}\\
 &=&b(V)-(-1)^{l(V)}.
\end{eqnarray*}
Therefore, $b(V)= (-1)^{l(V)}$ as desired. This completes the proof of Theorem \ref{Adams operators}.
\end{proof}

\section{Proof of Theorem \ref{PBW-basis}}
\label{proof-PBW}

This section is devoted to prove Theorem \ref{PBW-basis}. We consider it as a special case of a more general setting. The existence of such data, as this theorem asked for, has been established essentially in \cite{ZSL3} and later extended to a more general context in \cite{LZ,Zhou}. It was done constructively by means of a combinatorial method based on Lyndon words, which is originally developed by Kharchenko \cite{Kh}. In the sequel, we will show briefly how these desired datas are constructed and why they are workable. It is mainly for self-containedness and  reader's convenience. 

Throughout, let  $H:=\bigoplus_{\gamma\in \Gamma} H_\gamma$ be a connected $\Gamma$-graded algebra, where $\Gamma$ is a nonzero free abelian monoid. 
We fix a set $X$ of homogeneous generators of $H$ of nonzero degree.
We denote by $\langle X\rangle$ the set of words on $X$, by $1$ the empty word, and by $l(u)$ the length of a word $u$. Let $\deg:\langle X\rangle \to \Gamma$ be the homomorphism of monoids  assigning $x\in X$ to the degree of $x$ as an element of $H$.

Choose a well order $<_\Gamma$ on $\Gamma$ which is compatible with the addition of $\Gamma$ in the sense that if $\gamma_1<_\Gamma \gamma_2$ then $\gamma+\gamma_1<_\Gamma\gamma+\gamma_2$ for any $\gamma_1,\gamma_2, \gamma \in \Gamma$; also choose  a well order $<_\gamma$ on $X_\gamma:=\{~x\in X~|~\deg(x)=\gamma~\}$ for each  $\gamma\in \Gamma \backslash\{ 0\}$. Then define a partial order $<$ on $X$  as follows: for $x_1,x_2\in X$,
\begin{align*}\label{definition-deglex}
	x_1 < x_2\,\, \Longleftrightarrow \,\, \left\{
	\begin{array}{llll}
		\deg(x_1)<_\Gamma\deg(x_2), \quad \text{or} && \\
		\deg(x_1) =\deg(x_2)\, \text{ and } x_1 <_{\deg(x_1)} x_2.
	\end{array}\right. 
\end{align*}
It is easy to check that $<$ is a well order on $X$. The  \emph{pseudo-lexicographic order}  on $\langle X\rangle$ associated to $<$, which is denoted by $<_{\lex}$, is defined as follows: for $u,v \in \langle X\rangle$,
\begin{eqnarray*}
	\label{definition-deglex}
	u <_{\lex} v\,\, \Longleftrightarrow \,\, \left\{
	\begin{array}{llll}
		u=vw \text{ for some } w\neq 1,\quad \text{or}&& \\
		u=rxs,\, v=ryt,\, \text{ with } x,y\in X;\, x<y;  \, r,s,t\in \langle X\rangle.
	\end{array}\right.
\end{eqnarray*}
Note that it is a total order compatible with the concatenation of words from left but not from right.  For example if $x,y\in X$ with $x<y$, one has $y^2<_{\lex}y$ but $yx<_{\lex} y^2x$. 

\begin{lemma}\label{lex-restriction}
The restriction of $<_{\lex}$ to the set $\{ u\in \langle X\rangle ~|~ \deg(u)=\gamma \}$ is a well order for each $\gamma\in \Gamma$.	
\end{lemma}

\begin{proof}
It follows  by  essentially the same argument as that of Lemma \ref{well-founded-order}.	
\end{proof}

\begin{definition}
A word $u\in \langle X\rangle$ is called \textit{Lyndon}  if $u$ is nonempty and $wv <_{\lex} u $ for every factorization $u=vw$ with $v, w\neq 1$. The set of all Lyndon words on $X$ is denoted by $\L$.
\end{definition}

For a word  $u$ of length $\geq2$, define   $u_R\in \langle X\rangle $ to be  the  largest proper suffix of $u$ with respect to $<_{\lex}$ and define  $u_L\in \langle X\rangle $ by the decomposition $u=u_Lu_R$.  The pair of words
$$\sh(u):=(u_L,u_R)$$ is called  the {\em Shirshov factorization} of $u$. Note that $u$ is Lyndon if and only if $u_L$ and $u_R$ are both Lyndon and $u_R<_{\lex} u_L$ (see \cite[Proposition 5.1.3]{Lo} with an appropriate adjustment).

Let $k\langle X\rangle$ be the free algebra on $X$, which  has $\langle X\rangle$ as a linear basis. We consider $k\langle X\rangle$ as a connected $\Gamma$-graded algebra with the natural grading induced from that of $H$. Using the Shirshov factorization, we define the bracket map $[-]: \L\to k\langle X\rangle$ inductively by 
\begin{eqnarray*}
	[u]\,\, := \,\, \left\{
	\begin{array}{ll}
		~	u , & \text{if } u \in X \\
		~	[u_L][u_R]-[u_R][u_L], &  \text{if } u \not\in X.
	\end{array}\right.
\end{eqnarray*}
Note that $[u]$ is homogeneous of the same degree as that of $u$.

Let $\pi:k\langle  X\rangle \to H $ be the canonical homomorphism of $\Gamma$-graded algebras induced by the inclusion map $X\hookrightarrow H$. Note that $I:=\ker \pi$ is a homogeneous ideal of $k\langle X\rangle$.

\begin{definition}
	A word $u\in \langle X\rangle$ is called \emph{$I$-reducible} if $u$ is the leading word of some homogeneous element of $I$ of the same degree as that of $u$ with respect to $<_{\lex}$. Otherwise, it is called \emph{$I$-irreducible}. The set of all $I$-irreducible Lyndon words on $X$ is denoted by $\mathcal{N}_I$.
\end{definition}

\begin{definition}
Let $h_I:\L \to \mathbb{N}_\infty$ be the map given by 
\[
h_I(u):= \min\{ ~ n\geq 1 ~ | ~  u^n  ~ \text{is $I$-reducible} ~ \}.
\]
By convention, $h_I(u) :=\infty$ if there is no integer $n$ such that $u^n$ is $I$-reducible.  Note that $h_I(u)=1$ if  $u$ is $I$-reducible; and $h_I(u)\geq 2$ for every $I$-irreducible Lyndon word $u$. 
\end{definition}


\begin{proposition}\label{PBW-generators}
Let $\Xi:=\N_I$ and $z_\xi:=\pi([\xi])$ for each $\xi\in \Xi$. Let $\vartriangleleft$ and $h$ be the restriction of  $<_{\lex}$ and  $h_I$  to $\Xi$ respectively. Assume there is a homomorphism of $\Gamma$-graded algebras $\Delta:H\to H\otimes H$ such that $\Delta(z) \in z\otimes 1+1\otimes z +H_+\otimes H_+$ for any homogeneous element $z$ of nonzero degree. Then 
\begin{enumerate}
	\item $\{z_V\}_{V\in \mathcal{P}(\Xi,\vartriangleleft,h)}$ is a basis of $H$;
	\item $\Delta(z_\xi) \in 1\otimes z_\xi +z_\xi\otimes 1 +H[\xi, \vartriangleleft, h]_+\otimes H[\xi, \vartriangleleft, h]_+$ for each $\xi\in \Xi$;
	\item $z_\xi^{h(\xi)} \in H[\xi, \vartriangleleft, h]$ for each $\xi\in \Xi$ with $h(\xi) <\infty$;
	\item $z_\nu z_\mu \in z_\mu z_\nu + H[\nu, \vartriangleleft, h]$ for each pair $\mu, \nu \in \Xi$ with $\mu \vartriangleleft \nu$;
	\item if $k$ is of characteristic $0$ then $h(\xi)$ is infinity for each $\xi\in \Xi$; while if $k$ is of characteristic $p>0$ then $h(\xi)$ is either infinity or a power of $p$ for  each $\xi\in \Xi$;
	\item the restriction of $\vartriangleleft$ to the set $\{~ \xi\in \Xi~|~ \deg(z_\xi)=\gamma ~\}$ is a well order for each $\gamma\in \Gamma$.
\end{enumerate}
\end{proposition}

\begin{proof}
	Let $\tilde{\Delta}:k\langle X\rangle \to k\langle X\rangle\otimes k\langle X\rangle$ be a homomorphism of $\Gamma$-graded algebras that lifts $\Delta$ through $\pi:k\langle  X\rangle \to H $. So the following diagram commutes 	
	\begin{eqnarray*}
		\xymatrix{
			k\langle X\rangle \ar[r]^-{\tilde{\Delta}} \ar[d]_\pi & k\langle X\rangle\otimes k\langle X\rangle \ar[d]^{\pi\otimes \pi} \\
			H\ar[r]^-{\Delta} &  H\otimes  H.
		}
	\end{eqnarray*}
	It follows that $\tilde{\Delta}(I) \subseteq k\langle X\rangle \otimes I + I\otimes k\langle X\rangle$. Generally, $\tilde{\Delta}(x) \neq 1\otimes x+x\otimes 1$; and moreover $\tilde{\Delta}$ is not necessarily coassociative in the sense that $(\tilde{\Delta} \otimes \id ) \circ \tilde{\Delta} = (\id \otimes \tilde{\Delta}) \circ \tilde{\Delta}.$ However, due to the assumption on $\Delta$ and the peculiar choice of the well order $<$ on $X$, we may require that 
	\[
	\tilde{\Delta}(x) \in 1\otimes x+x\otimes 1 +k\langle X_{<x}\rangle_+ \otimes k\langle X_{<x}\rangle_+, 
	\]
	where $X_{<x}:=\{~ x'\in X~|~ x'<x ~\}$. That is to say $\tilde{\Delta}$ is a bounded comultiplication on $k\langle X\rangle$ in the sense of \cite[Definition 3.1]{Zhou}.	Therefore, one may apply \cite[Proposition 3.3]{Zhou} and \cite[Proposition 3.4]{Zhou}  in the special case that the concerned braiding on $kX$  is simply the flipping map of $kX\otimes kX$ (see loc. cit. for undefined notions). Combine these two propositions as well as the above commuting square, Parts (1)-(5) follow readily. The final part is an immediate consequence of Lemma \ref{lex-restriction}.
\end{proof}

\begin{proof}[Proof of Theorem \ref{PBW-basis}]
For a connected $\Gamma$-graded Hopf algebra $H$, it is easy to check that $\Delta_H(z) \in z\otimes 1+1\otimes z +H_+\otimes H_+$ for any homogeneous element $z$ of nonzero degree. The result then follows immediately by the above proposition.
\end{proof}

\begin{remark}
Let $H$ be a cocommutative connected $\Gamma$-graded Hopf algebra $H$. Assume the base field $k$ is of characteristic $0$.  Fix a homogeneous generating set $X$ that consists of primitive elements. Then the family $\{z_\xi\}_{\xi\in \Xi}$ given in Proposition \ref{PBW-generators} is a homogeneous basis of the primitive space $P(H)$ and the map $h:\Xi\to \mathbb{N}_\infty$ given there  takes the constant value $\infty$. However, the total order  given in Proposition \ref{PBW-generators} is generally different from the one given in Example \ref{cocommutative}. Indeed, take two Lyndon words $\mu, \nu \in \Xi$ such that $\nu$ is a proper suffix of $\mu$, namely $\mu=w\nu$ for some nonempty word $w$. Then $\nu<_{\lex} \mu$, so $\nu \vartriangleleft \mu$ in the sense of Proposition \ref{PBW-generators}. But $\deg(z_\nu)<\deg(z_\mu)$, so $\mu \vartriangleleft \nu$ in the sense of Example \ref{cocommutative}. 
\end{remark}

\section{Applications to the Hopf algebra of permutations}

In this section, we consider the Hopf algebra $\mathfrak{S}{\rm Sym}$ of permutations  by examining the ideas of the previous sections.  We refer to \cite{AgSo, GrRe, MaRe, PoRe} as  basic references. Throughout, we assume the base field $k$ is of characteristic $0$ and write $H:=\mathfrak{S}{\rm Sym}$ for simplicity.
	
Let $\mathfrak{S}=\bigcup_{m\geq 0} \mathfrak{S}_m$ be the set of all permutations of all $m\in \mathbb{N}$. We use  one-line notation for permutations, and regard them as words (i.e. finite sequences) on  $\mathbb{Z}_+=\{1<2<3<\cdots\}$. The unique element of $\mathfrak{S}_0$ is denoted by $\theta$.
Let $\prec$ be the restriction of the pseudo-lexicographic order on $\langle \mathbb{Z}_+\rangle$ to $\mathfrak{S}$.  For example, 
	$
	123 \prec 12 \prec 132 \prec 1 \prec 213 \prec 21 \prec 231 \prec 312 \prec 321.
	$
	
For any permutations $\sigma = i_1\cdots i_m \in \mathfrak{S}_m$ and $\tau=j_1\cdots j_n \in \mathfrak{S}_n$, define $\sigma\times \tau\in \mathfrak{S}_{m+n}$ by $$\sigma\times \tau := i_1\cdots i_m (j_1+m)\cdots (j_n+m).$$
For example $12\times 231 = 12453$. By convention, $\theta\times \sigma =\sigma\times \theta =\sigma$ for any $\sigma \in \mathfrak{S}$.

	\begin{definition}
	Let $m\geq 1$. A permutation $\sigma \in \mathfrak{S}_m$ is called 
	\begin{enumerate}
		\item \emph{connected} if $\sigma$ doesn't preserve $\{1,\ldots, i\}$ for any  $i\in \{1,\ldots, m-1\}$;
		\item \emph{Lyndon}	if $\sigma_2\times \sigma_1 \prec \sigma $ for any factorization $\sigma =\sigma_1\times \sigma_2$ with $\sigma_1,\sigma_2\neq \theta$. 
	\end{enumerate}
     Let $\mathfrak{C}$ (resp. $\mathfrak{L}$) be the set of all connected (resp. Lyndon) permutations of all $m\ge 1$.
	\end{definition}
Note that connected permutations are Lyndon by definition. In the following table, we list all connected and Lyndon permutations  $\sigma \in \mathfrak{S}_{m}$ for $m\leq 3$.
\[
\begin{tabular}{|c|c|c|c|c|}
	\hline
	$\sigma\in \mathfrak{S}_{m}$ & connected & Lyndon {\textbackslash  connceted}  & others  \\ \hline
	m=1 & 1 & $\emptyset$  & $\emptyset$ \\ \hline
	m=2 & 21  & $\emptyset$ &  12 \\ \hline
	m=3 & 231, 312, 321 & 213 &  123, 132 \\ \hline 
\end{tabular}
\]  
\vskip 2mm

Let $\langle\mathfrak{C}\rangle$ be the set of all words (i.e. finite sequences) on $\mathfrak{C}$.  There is a canonical bijection $$\psi: \langle\mathfrak{C}\rangle \to \mathfrak{S}, \quad (\sigma_1,\ldots, \sigma_r) \mapsto \sigma_1\times \cdots \times \sigma_r.$$ Let $\varphi:\mathfrak{S} \to \langle\mathfrak{C}\rangle$ be the inverse of $\psi$. Consider $\langle\mathfrak{C}\rangle$ as a totally ordered set by equipping it with the pseudo-lexicographic order  associated to the restriction of the order $\prec$ on $\mathfrak{S}$ to $\mathfrak{C}$.

\begin{lemma}\label{Permutation-word}
 The map $\varphi: \mathfrak{S} \to \langle\mathfrak{C}\rangle$ is an isomorphism of totally ordered sets.
\end{lemma}

\begin{proof}
Note that for any two connected permutations $\sigma, \tau$, if $\sigma \prec \tau$ then it is necessary that $\sigma$ doesn't contains $\tau$ as a proper prefix. The result then follows by a straightforward argument.
\end{proof}

\begin{lemma}\label{Lyndon-permutation}
A permutation $\sigma\in \mathfrak{S}$ is Lyndon if and only if $\varphi(\sigma)$ is Lyndon as a word on $\mathfrak{C}$.
\end{lemma}
 
\begin{proof}
It is a direct consequence of Lemma \ref{Permutation-word}. 
\end{proof}

For each Lyndon permutation $\sigma \in \mathfrak{L}$, define $\mathcal{T}_\sigma \in H$ inductively as follows:
\begin{eqnarray*}
	\mathcal{T}_\sigma \,\, := \,\, \left\{
	\begin{array}{ll}
		~	\mathcal{F}_\sigma, & \text{if } \sigma  \text { is connected} \\
		~  [\mathcal{T}_{\sigma_L}, \mathcal{T}_{\sigma_R}], &  \text{if } \sigma \text { is not connected},
	\end{array}\right.
\end{eqnarray*}
where $\sigma_L:= \psi(\varphi(\sigma)_L)$ and  $\sigma_R:= \psi(\varphi(\sigma)_R)$.  Let $h: \mathfrak{L} \to \mathbb{N}_\infty$ be the map with constant value $\infty$. The restriction of the order $\prec$ on $\mathfrak{C}$  to $\mathfrak{L}$ is denoted by the same symbol $\prec$. 

\begin{theorem}\label{PBW-permutation}
Consider the triple $(\{\mathcal{T}_\sigma\}_{\sigma\in \mathfrak{L}}, \prec, h)$ given as above, one has
\begin{enumerate}
	\item $\{\mathcal{T}_V\}_{V\in \mathcal{P}(\mathfrak{L},\prec,h)}$ is a basis of $H$;
	\item $\Delta(\mathcal{T}_\sigma) \in 1\otimes \mathcal{T}_\sigma + \mathcal{T}_\sigma\otimes 1 +H[\sigma, \prec,h]_+\otimes H[\sigma, \prec,h]_+$ for each $\sigma\in \mathfrak{L}$;
	\item $\mathcal{T}_\nu \mathcal{T}_\mu \in \mathcal{T}_\mu \mathcal{T}_\nu + H[\nu, \prec,h]$ for each pair $\mu, \nu \in \mathfrak{L}$ with $\mu \prec \nu$.
\end{enumerate}
\end{theorem}
\begin{proof}
By \cite[Theorem 2.1]{PoRe}, $H$ is freely generated by $X=\{\mathcal{F}_\sigma\}_{\sigma\in \mathfrak{C}}$. For the brevity of notation, we identify the word on $X$ with the word on $\mathfrak{C}$ in the natural way.  Note that each $\mathfrak{C}_m := \mathfrak{C} \cap \mathfrak{G}_m$ is a finite set. Moreover,  if $\sigma \in \mathfrak{C}_m$ and $\tau\in \mathfrak{C}_n$ with $m<n$ then $\sigma \prec \tau$. So the arguments taken in Section \ref{proof-PBW} apply. Since all Lyndon words in $\langle \mathfrak{C}\rangle$ are irreducible, the set $\Xi$ given in Proposition \ref{PBW-generators} is exactly the set of all Lyndon words in $\langle \mathfrak{C}\rangle$. By Lemma \ref{Permutation-word} and Lemma \ref{Lyndon-permutation}, it is easy to see that the map $\varphi: \mathfrak{G} \to \langle \mathfrak{C}\rangle$ restricts to an isomorphism of totally ordered sets $(\mathfrak{L},\prec) \xrightarrow{\cong} (\Xi,\vartriangleleft)$.  Moreover, one has $\mathcal{T}_\sigma = z_{\varphi(\sigma)}$  for each Lyndon permutation $\sigma\in \mathfrak{L}$. The result follows by Proposition \ref{PBW-generators}.
\end{proof}
	
By Lemma \ref{Permutation-word} and \cite[Proposition 1.4 (L5)]{ZSL3}, each permutation  $\sigma$ admits a unique decomposition 
	\[
	\sigma = \sigma_1\times \cdots \times \sigma_r, \quad r\geq 0,  \quad \sigma_1 \preceq \cdots \preceq \sigma_r\in \mathfrak{L}.
	\]
We call it the \emph{Lyndon decomposition} of $\sigma$. Taking Lyndon decomposition then defines a bijection
$\Phi: \mathfrak{S} \to \mathcal{P}(\mathfrak{L}, \prec, h)$. For a permutation $\sigma\in \mathfrak{G}$ which is not Lyndon, let
\[
\mathcal{T}_\sigma := \mathcal{T}_{\Phi(\sigma)} = \mathcal{T}_{\sigma_1} \cdots \mathcal{T}_{\sigma_r},
\]
where $\Phi(\sigma) = (\sigma_1,\ldots, \sigma_r)$.
By convention, $\mathcal{T}_\theta=1$ for the unique element $\theta\in \mathfrak{S}_0$. By Theorem \ref{PBW-permutation} (1), one obtains a third $\mathfrak{S}$-indexed homogeneous basis of $H$, namely
\[
\{\mathcal{T}_\sigma\}_{\sigma\in \mathfrak{S}},
\]
other than the fundamental  basis $\{\mathcal{F}_\sigma\}_{\sigma\in \mathfrak{S}}$ and the monomial basis $\{\mathcal{M}_\sigma\}_{\sigma\in \mathfrak{S}}$.

There are two total orders $\prec_L$ and $\prec_R$ on $\mathcal{P}(\mathfrak{L}, \prec, h)$, as defined in Section \ref{main-results}. We transfer them to total orders on $\mathfrak{S}$ via $\Phi$, which we use the same notation correspondingly. Take examples, 
\begin{align*}
&123 \prec_L 132 \prec_L 213 \prec_L 231 \prec_L 312 \prec_L 321;\\
&123 \prec_R 213 \prec_R 132 \prec_R 231 \prec_R 312 \prec_R 321.
\end{align*}
Not that if $\sigma\in \mathfrak{S}_m$ and $\tau\in \mathfrak{S}_n$ with $m<n$ then $\sigma \prec_L \tau$ and $\sigma \prec_R \tau$. In addition, by \cite[Lemma 1.5]{ZSL3}  it is not hard to see that $\sigma \prec_L \tau$ if and only if $\sigma \prec \tau$  provided that $m=n$. 

Let  $l(\sigma):=l(\Phi(\sigma))$ for each $\sigma\in \mathfrak{S}$. It is the number of factors in the Lyndon decomposition of $\sigma$.

\begin{theorem}
For each integers $m,n$ with $m\geq 0$,  and each permutation $\sigma \in \mathfrak{S}_m$, one has 
\[
\Psi_n(\mathcal{T}_\sigma)\in n^{l(\sigma)} \cdot \mathcal{T}_\sigma + {\rm Span} \big(\{~ \mathcal{T}_\mu~|~ \mu  \in \mathfrak{S}_m, ~ \mu \prec_R \sigma~\}\big ).
\]
Consequently, the  restrictions of the Adams operators to  $H_m$ are simultaneously upper triangularizable with respect to the ordered basis $\{\mathcal{T}_\sigma\}_{\sigma\in \mathfrak{S}_m}$, where $\mathfrak{S}_m$ is ordered by $\prec_R$.
\end{theorem}
\begin{proof}
	This is a direct consequence of Theorem \ref{PBW-permutation} and Theorem \ref{Adams operators}.
\end{proof}

\begin{example}
It is a continuation of Example \ref{counterexample}. The matrix of $\Psi_2|_{H_3}$ with respect to the ordered basis $\mathcal{T}_{123}, \mathcal{T}_{213}, \mathcal{T}_{132}, \mathcal{T}_{231}, \mathcal{T}_{312}, \mathcal{T}_{321}$ of $H_3$ is
\[
\left[ 
\begin{matrix}
8 & 0 & 2 & 1 & 1 & 0 \\
0 & 2 & 1 & -1 & 1 & 1 \\
0 & 0 & 4 & 0 & 0 & 2 \\
0 & 0 & 0 & 2 & 0 & 0 \\
0 & 0 & 0 & 0 & 2 & 0 \\
0 & 0 & 0 & 0 & 0 & 2
\end{matrix}
\right ]
\]
Indeed, by definition one has $\mathcal{T}_1=\mathcal{F}_1$, $\mathcal{T}_{12}=\mathcal{T}_{1}\mathcal{T}_{1}=\mathcal{F}_{12}+\mathcal{F}_{21}$,  $\mathcal{T}_{21}=\mathcal{F}_{21}$ and
\begin{align*}
	&\mathcal{T}_{123} = \mathcal{T}_1\mathcal{T}_1\mathcal{T}_1  = \mathcal{F}_{123} + \mathcal{F}_{132} + \mathcal{F}_{213} + \mathcal{F}_{231} + \mathcal{F}_{312} + \mathcal{F}_{321},\\
	&\mathcal{T}_{213} = [\mathcal{T}_{21}, \mathcal{T}_1] = -\mathcal{F}_{132} + \mathcal{F}_{213} + \mathcal{F}_{231} -\mathcal{F}_{312}, \\
	&\mathcal{T}_{132} = \mathcal{T}_1\mathcal{T}_{21} = \mathcal{F}_{132} + \mathcal{F}_{312} +\mathcal{F}_{321}, \\
	&\mathcal{T}_{231} =  \mathcal{F}_{231}, \quad \mathcal{T}_{312} =  \mathcal{F}_{312}, \quad \mathcal{T}_{321} =  \mathcal{F}_{321}.
\end{align*}
The  desired matrix  follows by using the formulas that obtained in Example \ref{counterexample}. 
\end{example}

\section*{Acknowledgments}

This work is supported by Shanghai Key Laboratory of Pure Mathematics and Mathematical Practice (Grant Nos. 22DZ2229014), the NSFC (Grant Nos. 12371039), and the K. C. Wong Magna Fund in Ningbo University. The authors thank the referee for his/her careful reading and valuable suggestions. Actually, the final section is based on one of the suggestions of the referee.


\end{document}